\documentclass[a4paper,12pt]{article}
\usepackage{amsmath,amssymb,amsthm,amsfonts}
\usepackage[mathscr]{euscript}
\usepackage{cite}
\newtheorem{thm}{Theorem}[section]

\newtheorem{defn}{Definition}[section]

\newtheorem{ex}{Example}[section]

\begin{document}
\lineskip 1.5em
\begin{center}
\vspace{1cm}
 {\bf \large On Polar Decomposition of Dual Matrices }
\vspace{0.8cm}

                             {\bf  V.Bhagyalakshmi}\\
                        Department of Mathematics\\
                        National  Institute of Technology Andhra Pradesh\\
                                Andhra Pradesh - 534101, India. \\
                                and\\

                          {\bf T. Kurmayya }\\
                        Department of Mathematics\\
                        National Institute of Technology Andhra Pradesh\\
                                Andhra Pradesh - 534101, India. \\

\end{center}
\begin{abstract}

A number of the form $a_s+a_d \epsilon$, where $a_s, a_d \in\mathbb{C}^{n \times n}$ and $\epsilon$ is an infinitesimal unit satisfying $\epsilon^2=0$, is called a dual number. A matrix with dual number entries is known as dual matrix. The objective of this article is to derive  a polar decomposition of dual matrices and to study the  existence and uniqueness conditions. To illustrate these results, some examples are constructed.

\end{abstract}

\vspace{1cm} {\bf Keywords:}\,\,Dual complex numbers; Dual matrix ; 
Unitary decomposition; Polar decomposition.\,\,\,

\vspace{.2cm} {\bf AMS Subject Classification:}\,\,\,\,  15A23 , 15A18
 
.\\
\thispagestyle{empty}

\newpage
\section{Introduction} 

Dual numbers were first introduced by \cite{wk1873preliminary} as a part of his extensive research on clifford algebras. They take the form $a=a_s+a_d \epsilon$ where $a_s$ and $a_d$ are real or complex numbers and $\epsilon$ is the infinitesimal unit satisfying $\epsilon^2=0$. Accordingly, we call $A=A_s+A_d \epsilon$ with $A_s$ and $A_s\in\mathbb{ C}^{m \times n}$ a dual complex matrix. The set of this kind of matrices is denoted as $D \mathbb {C}^{m \times n}$. Dual numbers and dual matrices have found applications in various fields such as robotics \cite{pradeep1989use,gu1987dual} , classical mechanics {\cite{wk1873preliminary ,pennestri2009linear}} and computer vision\cite{penunuri2019dual}. Moreover they are also linked with screw theory\cite{fischer2017dual}. For decompositions of dual matrices, some authors discussed singular value decomposition\cite{qi2022low,wei2024singular}, QR decomposition \cite{xu2024qr,pennestri2007linear}, UTV decomposition \cite{xu2024utv} and rank decomposition\cite{wang2023dual}. In this paper, we focus on Polar decomposition and Symmetric unitary polar decomposition of dual matrices.

Motivated by\cite{fassbender2008note} we introduce the factorization of a square(rectangular) dual complex matrix $A$ of the form
$$
\begin{aligned}
& A=P U \ldots \ldots \text { (1) and } \\
& A=S U \ldots \ldots \text { (2) } 
\end{aligned}
$$
where $P$ is positive semi definite, $S$ is complex symmetric and $U$ is unitary. We call these factorizations as Dual Polar Decomposition and Dual Symmetric Unitary Polar Decomposition (DSUPD) respectively. It is shown that the DSUPD exists for every matrix $A$ and it always non unique. Even the dual symmetric factor $S$ on be chosen in infinitely many ways. where as dual polar decomposition exists for every matrix $A$ and it is unique if $A$ is nonsingular. i.e the matrix $P$ is uniquely determined as the dual Hermitian positive semidefinite square root of $A A^*$ and $U$ is uniquely determined by (1) for a nonsingular $A$.

If $A$ is singular, then $A$ admits infinitely many representations of form (1). The dual polar decomposition and DSUPD can, in principle, be extended to rectangular matrices. In this article, we present the formulation of polar decomposition for both square and rectangular matrices. However, for the development and analysis of DSUPD, we restrict our attention to the case of square matrices.
More precisely, factorization (1) should be termed as left dual polar decomposition because $A$ can be factored as

$$
\begin{aligned}
A=W Q  \ldots \ldots \text { (3) }
\end{aligned}
$$
Where $W $ is dual unitary and $Q$ is dual Hermitian positive semidefinite. Again $Q$ is uniquely determined as the square root of $A^* A$. We call (3) a right dual polar decomposition.
Similarly we introduce right DSUPD and left DSUPD
i.e $ A=S U$ and $A=V R$
where $U, V$ are dual unitary while $S$,$R$ are dual complex symmetric and there is an intimate relationship between dual polar decomposition and dual symmetric unitary polar decomposition.

In recent years, there has been growing interest in extending classical decomposition results to more general frame works. Particularly those involves dual structures. The primary objective of this work is to investigate the existence and uniqueness properties of such decomposition. In this connection, the dual polar decomposition and DSUPD is obtained by extending the classical polar decomposition and SUPD frame work to incorporate dual matrices. Dual polar decomposition and DSUPD coincide when the field is real.\\
The paper is organized as follows . In section 2, we introduced basic notations, definitions and results . In section 3 and 4, we prove series of lemmas and the main theorems.

\section{Notations,  Definitions and Preliminaries}
In this section, we first briefly explain some of the terminologies used in this article. The set of real numbers, complex numbers, dual numbers, and dual complex numbers are denoted by $\mathbb{R}, \mathbb{C}, \mathbb{D}, \mathbb{D C}$. A dual number is called appreciable if its standard part is nonzero; otherwise, it is called infinitesimal. Suppose we have two dual complex numbers $a=a_s+a_d \epsilon$ and $b=b_s+b_d \epsilon$. Then their sum is $a+b=\left(a_s+b_s\right)+\left(a_d+b_d\right) \epsilon$, and their product is $a b=a_s b_s+ \left(a_s b_d+a_d b_s\right) \epsilon$. The multiplication of dual complex numbers is commutative. We recall the total order ' $\leq \prime$ 'over $\mathbb{D}$ defined by Qi et al. \cite{qi2022dual}. Suppose that ${a}=a_s+a_d \epsilon, {b}=b_s+b_d \epsilon \in \mathbb{D}$. Then,\\
(i) ${a}={b}$ if $a_s=b_s$ and $a_d=b_d$;\\
(ii) ${a}<{b}$ if $a_s<b_s$, or $a_s=b_s$ and $a_d<b_d$;\\
(iii) ${a} \leq {b}$ if $a_s<b_s$, or $a_s=b_s$ and $a_d \leq b_d$.\\
\noindent
In particular, a dual number ${a}=a_s+a_d \epsilon \in \mathbb{D}$ is positive, nonnegative, negative and nonpositive if ${a}>0, {a} \geq 0, {a}<0$, and ${a} \leq 0$, respectively. Denote $\mathbb{D}_{+}, \mathbb{D}_{++}$as the set of nonnegative and positive dual numbers, respectively. Dual number $a=a_s+a_d \epsilon$, its magnitude is defined as a nonnegative dual number
$$
|a|= \begin{cases}\left|a_s\right|+\operatorname{sgn}\left(a_s\right) a_d \epsilon, & \text { if } a_s \neq 0, \\ \left|a_d\right| \epsilon, & \text { otherwise } .\end{cases}
$$
For any dual number $a=a_s+a_d \epsilon$ and dual number $b=b_s+b_d \epsilon$ with $b_s \neq 0$, or $a_s=0$ and $b_s=0$, there is
$$
\frac{a_s+a_d \epsilon}{b_s+b_d \epsilon}= \begin{cases}\frac{a_s}{b_s}+\left(\frac{a_d}{b_s}-\frac{a_s}{b_s} \frac{b_d}{b_s}\right) \epsilon, & \text { if } b_s \neq 0 \\ \frac{a_d}{b_d}+c \epsilon, & \text { if } a_s=0, b_s=0\end{cases}
$$
where $c$ is an arbitrary number.
Later, when a dual number is nonnegative or positive, we say it is a nonnegative dual number or a positive dual number respectively. If we say a number is a nonnegative number or positive number, then that number should be a real number. We use 0, 0 , and $O$ to denote a zero number, a zero vector, and a zero matrix, respectively. In particular, for  a dual complex number ${a}=a_s+a_d \epsilon$, its conjugate is defined as ${a}^*=\bar{a}_s+\bar{a}_d \epsilon$, where $\bar{a}_s$ and $\bar{a}_d$ represent the complex conjugates of $a_s$ and $a_d$, respectively. We denote the set of all complex matrices of order $m \times n$ by $\mathbb{C}^{m \times n}$. For two complex matrices $A_s, A_d \in \mathbb{C}^{m \times n}, {A}=A_s+A_d \epsilon \in \mathbb{D C}^{m \times n}$ represents a dual complex matrix, where $\mathbb{D} \mathbb{C}^{m \times n}$ is the set of $m \times n$ dual complex matrices. The $\operatorname{rank}$ of $A \in \mathbb{C}^{m \times n}$ is denoted by $\operatorname{rank}(A), A^*=\left(\bar{a}_{i j}\right)^T$ and ${A}^*=A_s^*+A_d^* \epsilon$ denote the conjugate transpose of a complex matrix and dual complex matrix, respectively. A dual complex matrix is Hermitian if ${A}={A}^*$, and unitary if ${A}^* {A}={A} {A}^*=I_n$, where $I_n$ is $n \times n$ identity matrix. Apparently, $A \in \mathbb{D} \mathbb{C}^{n \times n}$ is unitary if and only if its column vectors form an orthonormal basis of $\mathbb{D C}^n$. Let $n \leq m $, we say that $A \in \mathbb{D C}^{m \times n}$ is partially unitary if its column vectors are unit vectors and orthogonal to each other.  \\
We begin this section by recalling the Eigenvalues of dual complex matrices.\\

\noindent
It is classical that an $n \times n$ real symmetric or complex Hermitian matrix has $n$ real eigenvalues, and this matrix is positive semidefinite (or definite, respectively) if and only if all of these $n$ eigenvalues are nonnegative (or positive, respectively). In 1997, Zhang\cite{zhang1997quaternions} extended this to quaternion Hermitian matrices. In 2023, Qi and Luo \cite{qi2021eigenvalues} further extended this to dual quaternion Hermitian matrices. This is actually also true for dual symmetric or dual complex Hermitian matrices. We now state these in our general frame.

Let $A \in \mathbb{D}\mathbb{C}^{n \times n}$ be a  Hermitian matrix and $\mathbf{x} \in \mathbb{D}\mathbb{C}^n$, we have $\left(\mathbf{x}^* A \mathbf{x}\right)^*=\mathbf{x}^* A \mathbf{x}$. This implies that  $\mathbf{x}^* A \mathbf{x}$ is a dual number,  by\cite{qi2022dual}, we may distinguish that $\mathbf{x}^* A \mathbf{x}$ is nonnegative, or positive, or not. If for all $\mathbf{x} \in \mathbb{D}\mathbb{C}^n, \mathbf{x}^* A \mathbf{x}$ is nonnegative, then we say that $A$ is positive semidefinite. If for all $\mathbf{x} \in \mathbb{D}\mathbb{C}^n$ and appreciable, $\mathbf{x}^* A \mathbf{x}$ is positive, then we say that $A$ is positive definite.\\
Let $A \in \mathbb{D}\mathbb{C}^{n \times n}, \mathbf{x} \in \mathbb{D}\mathbb{C}^n$ be appreciable, and $\lambda \in \mathbb{D}\mathbb{C}$. If
$$
A \mathrm{x}=\mathrm{x} \lambda,
$$
then $\lambda$ is called a right eigenvalue of $A$, with $\mathbf{x}$ as its corresponding right eigenvector. If
$$
A \mathbf{x}=\lambda \mathbf{x}
$$
where $\mathbf{x}$ is appreciable, i.e., $\mathbf{x}_s \neq \mathbf{0}$, then $\lambda$ is called a left eigenvalue of $A$, with a left eigenvector $\mathbf{x}$.\\
\noindent
We begin this section by recalling the unitary decomposition and singular value decomposition of a dual complex matrix.
 \begin{thm}(Theorem 4.4, \cite{qi2022low})
Suppose that $A=A_{s }+A_{{d}}\epsilon \in \mathbb{D} \mathbb{C}^{n \times n}$ is a Hermitian matrix. Then there are unitary matrix $U \in \mathbb{D} \mathbb{C}^{n \times n}$ and a diagonal matrix $\Sigma \in \mathbb{D}^{n \times n}$ such that $\Sigma=U^* A U$, where
$$
\Sigma \equiv \operatorname{diag}\left(\lambda_1+\lambda_{1,1} \epsilon, \cdots, \lambda_1+\lambda_{1, k_1} \epsilon, \lambda_2+\lambda_{2,1} \epsilon, \cdots, \lambda_r+\lambda_{r, k_r} \epsilon\right)
$$
with the diagonal entries of $\Sigma$ being $n$ eigenvalues of $A$,
$$
A \mathbf{u}_{i, j}=\left(\lambda_i+\lambda_{i, j} \epsilon\right) \mathbf{u}_{i, j},
$$
for $j=1, \cdots, k_i$ and $i=1, \cdots, r, U=\left(\mathbf{u}_{1,1}, \cdots, \mathbf{u}_{1, k_1}, \cdots, \mathbf{u}_{r, k_r}\right), \lambda_1>\lambda_2>\cdots>\lambda_r$ are real numbers, $\lambda_i$ is a $k_i$-multiple eigenvalue of $A_{s t}, \lambda_{i, 1} \geq \lambda_{i, 2} \geq \cdots \geq \lambda_{i, k_i}$ are also real numbers, $\sum_{i=1}^r k_i=n$. Counting possible multiplicities $\lambda_{i, j}$, the form $\Sigma$ is unique.
\end{thm}
\begin{thm}(Theorem 5.2 ,\cite{qi2022low})\label {Theorem 2}
Let ${A} \in \mathbb{D C}^{n \times m}$. Then, there exist dual complex unitary matrices ${V} \in \mathbb{D C}^{n \times n}$ and ${W} \in \mathbb{D C}^{m \times m}$ such that
$$
{A}={V} {\Sigma} {W}^*={V}\left[\begin{array}{cc}
{\Sigma}_t & O \\
O & O
\end{array}\right] {W}^*={V}\left[\begin{array}{cc}
{\Sigma}_1 & O \\
O & {\Sigma}_2
\end{array}\right] {W}^*
$$
where ${\Sigma}_1 \in \mathbb{D}^{r \times r}, {\Sigma}_1=\operatorname{diag}\left({\mu}_1, \ldots, {\mu}_r\right)$ and ${\Sigma}_2 \in \mathbb{D}^{n-r \times m-r}, {\Sigma}_2=\operatorname{diag}\left ({\mu}_{r+1}, \ldots, {\mu}_t\right.$, $0, \ldots, 0)$. Further, $r \leq t \leq \min \{n, m\}, {\mu}_1 \geq \ldots \geq{\mu}_r$ are positive appreciable dual numbers and ${\mu}_{r+1} \geq \ldots \geq {\mu}_t$ are positive infinitesimal dual numbers. Counting possible multiplicities of the diagonal entries, the form ${\Sigma}_t$ is unique.
\end{thm} 
We can also write
$$
{\Sigma}=\left[\begin{array}{cc}
{\Sigma}_t & O \\
O & O
\end{array}\right]=\left[\begin{array}{cc}
{\Sigma}_1 & O \\
O & {\Sigma}_2
\end{array}\right]=\left[\begin{array}{cc}
\Sigma_{1 s}+\Sigma_{1 d} \epsilon & O \\
O & \Sigma_{2 d} \epsilon
\end{array}\right]
$$
Let ${A}={V}\left[\begin{array}{cc}{\Sigma}_1 & 0 \\ 0 & \Sigma_{2 d} \epsilon\end{array}\right] W^*$ be the SVD of ${A}$. In the proof of Theorem 5.2, [15], ${W}$ is constructed using unitary decomposition (Theorem 4.4 \& 5.1, \cite{qi2022low}) of ${A}^* {A}$ and ${\Sigma}_1= \left[\begin{array}{cccc}
{\mu}_1 & \ldots & & 0 \\
0 & {\mu}_2 & & \vdots \\
\vdots & & \ddots & \\
0 & \ldots & & {\mu}_r
\end{array}\right]$
contain square root $\left(\sqrt{{a}}=\sqrt{a_s}+\frac{a_d}{2 \sqrt{a_s}}\right.$, where $\left.{a}=a_s+a_d \epsilon \geq 0, a_s \neq 0\right )$ of appreciable eigenvalue of ${A}^* {A}$. Now, the first $r$ column of ${V}=\left[{V}_1 {V}_2\right]$ is constructed.
\begin{defn}
In Theorem 2.2, the dual numbers $\mu_1, \cdots, \mu_t$ and possibly $\mu_{t+1}=\cdots=\mu_{\min \{m, n\}}=0$, if $t<\min \{n, m\}$, are called the singular values of $A$, the integer $t$ is called the rank of $A$, and the integer $r$ is called the appreciable rank of $A$. We denote the rank of $A$ by $\operatorname{Rank}(A)$, and the appreciable rank of $A$ by $\operatorname{ARank}(A)$.
\end{defn}
\section{Dual polar decomposition} 
The following theorem is required for the polar decomposition of dual matrices.
\begin{thm}
Let $\mathcal{F}$ be a given nonempty family of dual Hermitian matrices. There exists a unitary $U$ such that $U A U^*$ is diagonal for all $A \in \mathcal{F}$ if and only if $A B=B A$ for all $A, B \in \mathcal{F}$.
\end{thm}   
 \begin{thm}\label{Theorem 7.2.6}   
Let $A \in \mathbb{D}\mathbb{C}^{n \times n}$ be Hermitian and positive semidefinite, let $r=\operatorname{rank} A$, and let $k \in\{2,3, \ldots\}$.\\
(a) There is a unique  dual Hermitian positive semidefinite matrix $B$ such that $B^k=A$.\\
(b) There is a  dual polynomial $p$ with real coefficients such that $B=p(A)$. Consequently, $B$ commutes with any matrix that commutes with $A$.    
\end{thm}  
\begin{proof}
By unitary decomposition of dual matrices represent $A=U \Sigma U^*$ in which $U=\left[U_1   U_2\right]$ is unitary,\\
where 
$ 
U_1 \in D^{n \times r} and  \quad \Sigma=\operatorname{diag}\left(\sigma_1, \sigma_2, \ldots \sigma_{k_1}, \sigma_{k_1+1} \ldots \sigma_{k_2}, \sigma_{k_2+1}, \ldots \sigma_{k_3}\right. 
\left.\cdots \sigma_{k_{t-1}+1}, \sigma_{k_{t-1}+2} \cdots \sigma_r, \ldots, 0\right)
$
where 
$$
\begin{gathered}
\sigma_1=\lambda_1+\lambda_{1,1} \epsilon\\
\sigma_2=\lambda_1+\lambda_{1,2} \epsilon \\
\vdots\\
\sigma_{k_1}=\lambda_1+\lambda_{1, k_1 \epsilon} \\
\vdots \\
\sigma_{k_2}=\lambda_2+\lambda_{2, k_2 \epsilon} \\
\vdots \\
\sigma_r=\sigma_{k_t}=\lambda_t+\lambda_{t, k_t \epsilon} .
\end{gathered}
$$
the diagonal entries of $\Sigma$ being $n$ eigenvalus of $A$ and $\lambda_1, \ldots \lambda_s$ are positive\\
Define $B=U \Sigma^{1 / k} U^*$ in which
$
\Sigma^{1 / k}=\operatorname{diag}\left(\left(\lambda_1+\lambda_1, \epsilon\right)^{1 / k}, \cdots\left(\lambda_1+\lambda_{1, k_1 }\epsilon\right)^{1 / k},\left(\lambda_2+\lambda_2, \epsilon\right)^{1 / k} \cdots \cdot \cdot\right. \\
\left.\left(\lambda_s+\lambda_{t, k_t} \epsilon\right)^{1 / k}, \cdots \cdot 0\right)
$ 
and the unique nonnegative $k^th$ root is taken in each case.\\
Consider 
$$
\begin{aligned}
B^k & =\left(U \Sigma^{V / k} U^*\right)\left(U \Sigma^{1 / k} U^*\right) \cdots\left(U \Sigma^{1 / k} U^*\right) \\
& =\left(U \Sigma^{1 / k} U^*\right)^k \\
& =U \Sigma U^* \\
& =A .
\end{aligned}
$$
We have
$$
\begin{aligned}
B & =U \Sigma^{1 / k} U^* \\
B^*& =\left(U \Sigma^{1 / k} U^*\right)^*\\
& =U \Sigma^{1 / k} U^*\\
& =B
\end{aligned}
$$
this implies $B$ is Hermitian and positive semidefinite since eigenvalues of $B$ is some as eigenvalues of $\Sigma^{1 / k}$. \\
Let p be a dual polynomial such that $$ p\left(\lambda_i+\lambda_{i, j} \epsilon\right)=\left(\lambda_i+\lambda_{i, j} \epsilon\right)^{1 / k},$$
$i=1, \ldots t$ and $p(0)=0$ if $t<n$ , which shows that p has real dual coefficients. Then $p(\Sigma)=a_0+a_1 \Sigma+a_2 \Sigma^2+\cdots+a_t \Sigma^t+\cdots+a_n \Sigma^n$ , 
Where \( a_0, a_1, \ldots a_n, \Sigma \) are dual
$p\left(\Sigma_s+\Sigma_d \epsilon\right)=a_0 I+a_1 \left[\begin{array}{cccc}\left(\lambda_1+  \lambda_{1,_1 }\epsilon\right) & O & \cdots & O \\ O & \left(\lambda_2+ \lambda_{2, k_2}  \epsilon\right) & \ddots & \vdots \\ \vdots & \ddots & \ddots & O \\ O & \cdots & O & \left(\lambda_t+\lambda_{t, k_t}  \epsilon\right) \end{array}\right] +\cdots+\\a_n\left[\begin{array}{cccc}\left(\lambda_1+  \lambda_{1,_1 }\epsilon\right) & O & \cdots & O \\ O & \left(\lambda_2+ \lambda_{2, k_2}  \epsilon\right) & \ddots & \vdots \\ \vdots & \ddots & \ddots & O \\ O & \cdots & O & \left(\lambda_t+\lambda_{t, k_t}  \epsilon\right) \end{array}\right]^n $

  $=\left[\begin{array}{cccc}p\left(\lambda_1+  \lambda_{1,_1 }\epsilon\right) & O & \cdots & O \\ O & p\left(\lambda_2+ \lambda_{2, k_2}  \epsilon\right) & \ddots & \vdots \\ \vdots & \ddots & \ddots & O \\ O & \cdots & O & p\left(\lambda_t+\lambda_{t, k_t}  \epsilon\right) \end{array}\right]$

$=\left[\begin{array}{cccc}\left(\lambda_1+  \lambda_{1,_1 }\epsilon\right)^{1 / k} & O & \cdots & O \\ O & \left(\lambda_2+ \lambda_{2, k_2}  \epsilon\right)^{1 / k} & \ddots & \vdots \\ \vdots & \ddots & \ddots & O \\ O & \cdots & O & \left(\lambda_t+\lambda_{t, k_t}  \epsilon\right)^{1 / k} \end{array}\right]$

$=\Sigma^{1 / k}$  \\ and \\
 $\begin{aligned} p(A) & =p\left(U \Sigma U^*\right) \\ & =a_0+a_1\left(U \Sigma U^*\right)+\cdots+a_t\left(U \Sigma U^*\right)^t+\cdots+a_n\left(U \Sigma U^*\right)^n \\ & =U\left[a_0 I+a_1 \Sigma+\cdots+a_t \Sigma^t+\cdots+a_n \Sigma^n\right] U^* \\ & =U P(\Sigma) U^* \\ & =U \Sigma^{1 / K} U^* \\ & =B .\end{aligned}$ , \\which verifies (b). Now show that B commutes with any matrix that commute with A and it is unique Hermitian positive semidefinite matrix. If $C$ is a positive semidefinite dual Hermitian matrix such that  $ C^k$ = $A$ , then $B=p(A)=p(C^k)$ and hence $B$ commutes with $C$. Theorem 3.1 ensures that there is a dual unitary $U$ that simultaneously diagonalizes $B$ and $C$ ,  so $B=U\Sigma_1 U^*$ and $C=U \Sigma_2 U^*$, in which $\Sigma_1, \Sigma_2 \in D_n$ are nonnegative diagonal. 
 Since $B^k=A$ and $C^k=A$   
i.e $B^k=A=C^k$  \\
$\left(U \Sigma_1 U^*\right)^k=\left(U \Sigma_2 U^*\right)^k $\\
 $U \Sigma_1^k U^*=U \Sigma_2^k U^* $, \\  we deduce that 
$ \Sigma_1^k=\Sigma_2^k. $\\
 Uniqueness of the nonnegative $k$ th root of a nonnegative number implies that\\
$\Sigma_1=\left(\Sigma_1^k\right)^{1 / k}=\left(\Sigma_2^k\right)^{1 / k}=\Sigma_2$ ,
 so $B=C$. \\
\end{proof}
We now state the Polar decomposition theorem for dual complex matrices.
\begin{thm}   
Let $A \in \mathbb{D}\mathbb{C}^{n \times m}$
\begin{enumerate}
\item If $n<m$, then $A=P U$, in which $P \in  \mathbb {D} \mathbb {C}^{n\times n}$ is a dual positive semidefinite and $U \in \mathbb {D} \mathbb {C}^{n\times m}$ has dual orthonormal rows. The factor $P=\left(A A^*\right)^{1 / 2}$ is uniquely determined; it is a dual polynomial in $A A^*$. The factor $U$ is uniquely determined if $\operatorname{rank} A=n$.
\item If $n=m$, then $A=P U=U Q$, in which $P, Q \in \mathbb {D} \mathbb {C}^{n \times n}$ are positive semidefinite and $U \in \mathbb {D}\mathbb {C}^{n \times n}$ is unitary. The factors $P=\left(A A^*\right)^{1 / 2}$ and $Q=\left(A^* A\right)^{1 / 2}$ are uniquely determined; $P$ is a polynomial in $A A^*$ and $Q$ is a polynomial in $A^* A$. The factor $U$ is uniquely determined if $A$ is nonsingular.
\item If $n>m$, then $A=U Q$, in which $Q \in  \mathbb {D}\mathbb {C}^{m \times m}$ is positive semidefinite and $U \in \mathbb {D}\in \mathbb {C}^{n\times m}$ has orthonormal columns. The factor $Q=\left(A^* A\right)^{1 / 2}$ is uniquely determined; it is a polynomial in $A^* A$. The factor $U$ is uniquely determined if rank $A=m$.
\end{enumerate}
\end{thm}
\begin{proof}
By Theorem \ref{Theorem 2} we adopt the notation of singular value decomposition for dual matrices which ensures  that there are dual unitary matrices $V \in D_{n \times n}$ and $ W \in D_{m \times m}$  and a dual non negative diagonal matrix  $\Sigma \in D_{n \times m}$ with a special structure such that $A=V \Sigma W^*$.
Let $t =\min \{n, m\}$ and let $\Sigma_t \in D_{t \times t}$ be the diagonal matrix taking the form
$$
\Sigma_t=\operatorname{diag}\left(\mu_1 \cdots \mu_r \cdots \mu_t\right)
$$
$r \leq t \leq \min \{m, n\}, \quad \mu_1 \geq \mu_2 \geq \ldots \geq \mu_r$ are  positive appreciable dual numbers and $\mu_{r+1} \geq \ldots \geq \mu_t$ are positive infinitesimal dual numbers.\\
\noindent
\text{\bf $(1)$} Let $W=\left[\begin{array}{ll}W_1 & W_2\end{array}\right]$; in which $W_1 \in D_{m, n}$. Then \\
   $\begin{aligned} 
 A & =V\Sigma W^* \\[-1cm] 
 & \left.=V[\Sigma_n \ 0\right]\left[\begin{array}{l}W_1^* \\ W_2^*\end{array}\right] \\[-1cm] 
 & =V \Sigma_n W_1^* \\ 
 & =\left(V \Sigma_n V^*\right)\left(V W_1^*\right)
 \end{aligned}$\\
in which $P=V \Sigma_n V^*$ is positive semidefinite since P is unitarily similar to $ \Sigma_n.$ The eigenvalues of $ \Sigma_n $are same as eigenvalues of P and $U=V W_1^*$ has orthonormal rows.\\
Since $P^2=V \Sigma_n \Sigma_n V^*=V \Sigma \Sigma^T V^*= \left(V \Sigma W^*\right)\left(W \Sigma^T V^*\right)=A A^*, $ \\
By Theorem \ref {Theorem 7.2.6}, P is uniquely determined as the (polynomial) positive semidefinite square root of $A A^*$  .\\
If rank $A=n$, i.e it has n nonzero eigenvalues then $\Sigma_n$ and $P$ are positive definite, so $U=P^{-1} A$ is uniquely determined.\\
\text{\bf $(2)$} Let $\Sigma=\Sigma_n$. We have\\
 $\begin{aligned}
A & =V \Sigma W^* \\
& = V\Sigma V^*V W^* \\
& =\left(V \Sigma V^*\right)\left(V W^*\right) \\
& = PU  \\ and 
\end{aligned}$ \\ 
$\begin{aligned}
A & =V \Sigma W^* \\
& = V W^* W \Sigma W^* \\
& =\left(V W^*\right)\left(W \Sigma W^*\right) \\
& = UQ
\end{aligned}$ \\
So if we let $ P=V \Sigma V^* , Q= W\Sigma W^* and  U = V W^* $ then we have factorizations of the required form.\\
$\begin{aligned} P^2 & =\left(V \Sigma V^*\right)\left(V \Sigma V^*\right) \\ & =V \Sigma \Sigma^{\top} V^* \\ & =\left(V \Sigma W^*\right)\left(W \Sigma^{\top} V^*\right) \\ & =A A^*\end{aligned}$ \\ and \\
$\begin{aligned} Q^2 & =\left(W \Sigma W^*\right)\left(W \Sigma W^*\right) \\ & =W \Sigma^{\top} \Sigma W^* \\ & =\left(W \Sigma^{\top} V^*\right)\left(V \Sigma W^*\right) \\ & =A^* A .\end{aligned}$\\
By Theorem \ref {Theorem 7.2.6}, P and Q are uniquely determined as the respective  (polynomial) positive semidefinite square root of $A A^*$ and $ A A^* $  .\\
If $A$  is nonsingular i.e $A^{-1} $ exist . We have $A =PU=UQ $ where P and Q are positive definite since $A$ is nonsingular then $ U=P^{-1} A = A Q^{-1} $ is uniquely determined.\\
\text{\bf $(3)$} Let $V=\left[\begin{array}{ll}V_1 & V_2\end{array}\right]$; in which $V_1 \in D_{n, m}$. Then \\ 
$\begin{aligned} 
A & =V \Sigma W^* \\[-0.5cm]
&=\left[V_1  V_2\right]\left[\begin{array}{c}\Sigma_m \\ 0\end{array}\right] W^* \\[-0.5cm]
&=V_1 \Sigma_m W^* \\
& =\left(V_1 W^*\right)\left(W \Sigma_m W^*\right) \\ 
& =U Q
\end{aligned}$\\
in which $Q=W \Sigma_m W^*$ is positive semidefinite since Q is unitarily similar to $ \Sigma_m.$ The eigenvalues of $ \Sigma_m $are same as eigenvalues of Q  and $U=V_1 W^*$ has orthonormal rows.\\
Since $Q^2=W\Sigma_m \Sigma_m W^*=W \Sigma_m \Sigma_m^T W^*= \left(W \Sigma_m V^*\right)\left(V \Sigma_m^T W^*\right)= A^*A, $ \\
By Theorem \ref {Theorem 7.2.6}, Q is uniquely determined as the (polynomial) positive semidefinite square root of $ A^*A$  .\\
If rank $A=m$, i.e it has m nonzero eigenvalues then $\Sigma_m$ and $Q$ are positive definite, so $U=Q^{-1} A$ is uniquely determined.
\end{proof} \noindent
\textbf{NOTE} :  From the above theorem , we call factorization $(a)$ left dual polar decomposition and $(c)$ right dual polar decomposition.\\
We illustrate the above theorem with the following example.
 \begin{ex}
  Let $B=\begin{bmatrix}
            1& \epsilon &1\\
            \epsilon&\epsilon&\epsilon
        \end{bmatrix}\in \mathbb{D}^{2 \times 3}$. Then, 
           the unitary decomposition of $A=B^*B=\begin{bmatrix}
            1 & \epsilon & 1\\
            \epsilon & 0 & \epsilon\\
            1 & \epsilon & 1 
        \end{bmatrix}$, Applying unitary decomposition on $A_{s}$, we have $$S^*A_{s}S=\begin{bmatrix}
        2&0& 0\\
        0&0&0\\
        0&0&0 
    \end{bmatrix}, \text{where}~S=\begin{bmatrix}
        \frac{1}{\sqrt{2}} & 0 & \frac{1}{\sqrt{2}}\\
         0 & 1 & 0 \\
         \frac{1}{\sqrt{2}} & 0 & -\frac{1}{\sqrt{2}}
    \end{bmatrix}.$$
 Now, $$M=S^*A_{s}S+S^*A_{d}S\epsilon=\begin{bmatrix}
        2 & \sqrt{2} \epsilon & 0 \\
\sqrt{2} \epsilon & 0 & 0 \\
0 & 0 & 0
    \end{bmatrix}.$$   
From $M$, we obtain $\hat P=\begin{bmatrix}
        1 & \frac{1}{\sqrt{2}} \epsilon & 0 \\
-\frac{1}{\sqrt{2}} \epsilon & 1 & 0 \\
0 & 0 & 1
    \end{bmatrix}$. Thus, the unitary decomposition of $A$ is $$\hat W^*A \hat W=\begin{bmatrix}
        2&0& 0\\
        0&0&0\\
        0&0&0 
    \end{bmatrix}, \text{ where }\hat W =(\hat PS)^*=\begin{bmatrix}
       \frac{1}{\sqrt{2}} & -\frac{\epsilon}{2} & \frac{1}{\sqrt{2}} \\ \frac{\epsilon}{\sqrt{2}} & 1 & 0 \\ \frac{1}{\sqrt{2}} & -\frac{\epsilon}{2} & -\frac{1}{\sqrt{2}}
    \end{bmatrix}.$$


  Now, $\hat{V}_{1}=B\hat W_{1}\Sigma_{1}^{-1} =\begin{bmatrix}
            1 \\
            {\epsilon}
        \end{bmatrix}$, where $\hat W_{1}=\begin{bmatrix}
            \frac{1}{\sqrt{2}}\\
            -\frac{\epsilon}{\sqrt{2}}\\
            \frac{1}{\sqrt{2}}  
        \end{bmatrix}$ and $\Sigma_{1}=\begin{bmatrix}
            \sqrt{2}
        \end{bmatrix}$. Taking $ \hat V _{2}=\begin{bmatrix}
           {-\epsilon}\\
           1
        \end{bmatrix}$ such that $\hat V =\begin{bmatrix}
            1 & {-\epsilon}\\
{\epsilon} & 1
        \end{bmatrix}$ is unitary.Then, $$ \hat V^*B \hat W= \begin{bmatrix} \sqrt {2} & 0&0 \\ 0 &  \epsilon & 0 \end{bmatrix} 
        =\begin{bmatrix}
            \Sigma_{1}&0\\
            0&G\epsilon
        \end{bmatrix},~ \text{where}~ G=\begin{bmatrix}
            1 & 0
        \end{bmatrix}.$$ 
         By singular value decomposition of $G$, we obtain $X^*GY=\begin{bmatrix}
            1 & 0
        \end{bmatrix}$, where $X=\begin{bmatrix}
            1
        \end{bmatrix}$ and $Y=\begin{bmatrix}
            1 & 0\\
            0& 1
        \end{bmatrix}$.
  Thus, the singular value decomposition of $B$ is $$V^*BW=\begin{bmatrix}
        \sqrt{2} &0 & 0\\
        0 & {\epsilon} & 0
    \end{bmatrix},$$ where $$V=\hat V \begin{bmatrix}
        1&0\\
        0&X
    \end{bmatrix}=\begin{bmatrix}
        1 & {-\epsilon}\\
       {\epsilon}&1
    \end{bmatrix}$$ and $$W =\hat W\begin{bmatrix}
        1 & 0\\
        0 & Y
    \end{bmatrix}=\begin{bmatrix}
       \frac{1}{\sqrt{2}}&\frac{-\epsilon}{2} &\frac{1}{\sqrt{2}}\\
            \frac{\epsilon}{\sqrt{2}} & 1 & 0\\
            \frac{1}{\sqrt{2}} & \frac{-\epsilon}{2} & \frac{-1}{\sqrt{2}}
    \end{bmatrix}.$$
    Now we have to find the polar decomposition with the help of SVD.\\
    $ B= PU $ ,  \text{where}~ $P = V\Sigma_{2\times 2} V^* $ =  $\begin{bmatrix}
        \sqrt{2} & \sqrt{2} \epsilon \\
         \sqrt{2} \epsilon & \epsilon
    \end{bmatrix}$ and \\ $ U = V W_1^* $ = $\begin{bmatrix}
    \frac{1}{\sqrt{2}} & \frac{\epsilon}{\sqrt{2}}-\epsilon & \frac{1}{\sqrt{2}} \\ 
    \frac{\epsilon}{\sqrt{2}}-\frac{\epsilon}{2} & 1 & \frac{\epsilon}{\sqrt{2}}-\frac{\epsilon}{2}
    \end{bmatrix}$ \\
    Substitute P and U values
    $$ 
    PU  = \begin{bmatrix}
        \sqrt{2} & \sqrt{2} \epsilon \\
         \sqrt{2} \epsilon & \epsilon
         \end{bmatrix} 
    \begin{bmatrix}
    \frac{1}{\sqrt{2}} & \frac{\epsilon}{\sqrt{2}}-\epsilon & \frac{1}{\sqrt{2}} \\ 
    \frac{\epsilon}{\sqrt{2}}-\frac{\epsilon}{2} & 1 & \frac{\epsilon}{\sqrt{2}}-\frac{\epsilon}{2}
    \end{bmatrix} = \begin{bmatrix} 
     1 & \epsilon & 1 \\ \epsilon &\epsilon & \epsilon
     \end{bmatrix} = B 
    $$\\
 \end {ex}

\section {Symmetric unitary polar decomposition}
\begin{thm}
Let $A \in \mathbb{D}\mathbb{C}^{n \times n}$.\vspace{0.2cm}There exist a unitary matrix $U\in \mathbb{D}\mathbb{C}^{n \times n}$ and a complex symmetric $S\in \mathbb{D}\mathbb{C}^{n \times n}$ \vspace{0.2cm}such that $A=SU$. Furthermore,both left and right DSUPD exist for every matrix A and is always non unique.
\end{thm}
\begin{proof}
We begin by considering the SVD of $A$
$$
\begin{aligned}
A & =V \Sigma W^* \\
& =V \Sigma V^{\top} \bar{V} W^* \\
& =\left(V \Sigma V^{\top}\right)\left(\bar{V} W^*\right) \\
& =S U
\end{aligned}
$$
and observe that
$$
S=V \Sigma V^{\mathrm{T}}
$$
is a complex symmetric matrix, while
$$
U=\bar{V} W^*
$$
is unitary.This consideration establishes the existence of a \vspace{0.2cm}
left DSUPD for every matrix A.\vspace{0.2cm} The existence of a right DSUPD can be shown similarly.
Now from DSUPD we have 
$S=V \Sigma V^{\top}$. Replace $V$ by $V_1=V D$ where $D$ is diagonal matrix
we get 
$$
\begin{aligned}
S_1 &=V_1 \Sigma V_1^{\top}\\
& =(U D) \Sigma (V D)^{\top} \\
& =V D \Sigma D^{\top} U^{\top} \\
& =V \Sigma D^2 V^{\top}\\
& =V \Sigma D^2 V^{\top}
\end{aligned}
$$
 which is different from $S$
unless $D^2=I$. We conclude that there are infinitely  many left DSUPDs.The same is true of right DSUPs.
\end{proof}\noindent
The following example illustrates Theorem 3.3 and 4.1
\begin{ex}   
Consider a dual matrix $A=A_{s}+A_{d}\epsilon\in \mathbb{D}^{2 \times 2}$, where 
    $$A_{s}=\begin{bmatrix}
        1&0\\
        0&0
    \end{bmatrix} \text{and}~~A_{d}=\begin{bmatrix}
        0&1\\
        1&1
    \end{bmatrix}.$$ Then, $$A=V\Sigma W^T=\begin{bmatrix}
        1&-\epsilon\\
        \epsilon&1
    \end{bmatrix}\begin{bmatrix}
        1&0\\
        0&\epsilon
    \end{bmatrix}\begin{bmatrix}
        1&-\epsilon\\
        \epsilon&1
    \end{bmatrix}.$$ 
    
\textbf{Dual polar decomposition}
  $$P=V\Sigma V^*=\begin{bmatrix}
       1&-\epsilon\\
        \epsilon&1
    \end{bmatrix}\begin{bmatrix}
        1&0\\
        0&\epsilon     
    \end{bmatrix}\begin{bmatrix}
         1&\epsilon\\
        -\epsilon&1
    \end{bmatrix} $$
    $$U=\bar{V} W^*=\begin{bmatrix}
       1&-\epsilon\\
        \epsilon&1
    \end{bmatrix}\begin{bmatrix}

         1&\epsilon\\
        -\epsilon&1
    \end{bmatrix}$$

Hence $PU=A$

 \textbf{Dual symmetric unitary polar decomposition}
  $$S=V\Sigma V^T=\begin{bmatrix}
       1&-\epsilon\\
        \epsilon&1
    \end{bmatrix}\begin{bmatrix}
        1&0\\
        0&\epsilon     
    \end{bmatrix}\begin{bmatrix}
         1&\epsilon\\
        -\epsilon&1
    \end{bmatrix} $$
    $$U=\bar{V} W^*=\begin{bmatrix}
       1&-\epsilon\\
        \epsilon&1
    \end{bmatrix}\begin{bmatrix}

         1&\epsilon\\
        -\epsilon&1
    \end{bmatrix}$$

Hence $SU=A$
\end{ex} \noindent
\textbf{Observation} : 
From the above example, we observe that over the field of real numbers, both the dual polar decomposition and the DSUPD yield the same result
\bibliography{references}
\bibliographystyle{elsarticle-num} 
\end{document}